\documentclass[12pt]{amsart}
\usepackage{amssymb, amscd}

\newtheorem{Lemma}             {Lemma}
\newtheorem{Corollary}  [Lemma]{Corollary}

\newtheorem{Example}    [Lemma]{Example}
\newtheorem{Theorem}    [Lemma]{Theorem}

\title[Counting real characters]{Counting real characters after Gallagher}

\author{John C. Murray}
\address{Department of Mathematics and Statistics\\National University of Ireland Maynooth\\IRELAND}
\email{John.Murray@maths.nuim.ie}

\date{January 3, 2023}

\begin{document}
\begin{abstract}
Let $G$ be a finite group, let $N$ be a normal subgroup of $G$ and let $\theta$ be an irreducible character of $N$. We count the real irreducible characters of $G$ lying over $\theta$. 
\end{abstract}

\maketitle

\section{Introduction}

Let $G$ be a finite group and let $\theta$ be an irreducible character of a normal subgroup $N$ of $G$. Recall that an irreducible character $\chi$ of $G$ is said to lie over $\theta$ if $\chi$ is a constituent of the induced character $\theta{\uparrow^G}$. We first remind the reader of P.~X.~Gallagher's method for counting the irreducible characters of $G$ which lie over $\theta$. To this end let $G_\theta$ be the inertia group of $\theta$ in $G$. If $x\in G_\theta$, then $\theta$ extends to an irreducible character $\theta_x$ of $N\langle x\rangle$. Set $\overline G=G/N$. Then $\overline x$ has centralizer $\operatorname{C}_G(\overline x):=\{g\in G\mid \overline x^g=\overline x\}$ in $G$. We say that $\overline x$ is good for $\theta$ if $\theta_x^c=\theta_x$, for all $c\in\operatorname{C}_G(\overline x)$. This does not depend on the choice $\theta_x$ of extension. Also $\overline x\in\overline G_\theta$ is good for $\theta$ if and only if all of its $\overline G_\theta$-conjugates are good for $\theta$. So we can refer to the $\theta$-good conjugacy classes of $\overline G_\theta$.

\medskip
{\bf Gallagher's Theorem.}\label{T:Gallagher}
{\em The number of irreducible characters of\/ $G$ lying over $\theta$ equals the number of\/ $\theta$-good conjugacy classes of\/ $G_\theta/N$.}
\medskip

In this paper we modify Gallagher's proof to count the real-valued irreducible characters of $G$ which lie over $\theta$. If $G$ has such a character $\chi$, then $\chi$ also lies over the dual character $\overline\theta$. So we can and do assume from now on that $\overline\theta$ is $G$-conjugate to $\theta$.

The extended inertia group of $\theta$ in $G$ is the group
$$
G_\theta^*:=\{g\in G\mid \theta^g=\theta\mbox{ or }\overline\theta\}.
$$
Recall also that an element is said to be real in $G$ if it is conjugate to its inverse in $G$.

Let $\overline x\in\overline G_\theta$ be good for $\theta$, as in Gallagher's Theorem. We assign an indicator $\sigma(\overline x)=0,\pm1$ to $\overline x$, as follows. Suppose first that there is $e\in G_\theta^*$ such that $(\overline x,\theta)^e=(\overline x^{-1},\overline\theta)$. Then we set $\sigma(\overline x)=+1$ if $\theta_x^e=\overline\theta_x$ and $\sigma(\overline x)=-1$ if $\theta_x^e\ne\overline\theta_x$. If no such $e$ exists (in particular if $\overline x$ is not real in $\overline G_\theta^*$), we set $\sigma(\overline x)=0$. Then $\sigma(\overline x)$ does not depend on the choice of $e$. Moreover $\sigma(\overline y)=\sigma(\overline x)$, for all $\overline y\in\overline G$ which are $\overline G$-conjugate to $\overline x$. So we can refer to each $\theta$-good conjugacy class of $G_\theta$ as being of $\sigma_+$, $\sigma_-$ or $\sigma_0$ type. Our main result is:

\begin{Theorem}\label{T:main}
The number of real irreducible characters of\/ $G$ lying over $\theta$ equals the number of\/ $\theta$-good conjugacy classes of\/ $G_\theta/N$ of\/ $\sigma_+$ type minus the number of\/ $\theta$-good conjugacy classes of\/ $G_\theta/N$ of\/ $\sigma_-$ type.
\end{Theorem}

In particular, the number of real irreducible characters of $G$ lying over $\theta$ depends only on the elements of $\overline G_\theta$ which are good for $\theta$ and which are real in $\overline G_\theta^*$.

Recall that the Frobenius-Schur indicator $\epsilon(\theta)$ takes the values $\{+1,-1,0\}$ and is related to the realisability of $\theta$ over ${\mathbb R}$. For each $t\in G$, with $t^2\in N$, there is a related indicator $\epsilon_{\overline t}(\theta)$ called the Gow-indicator, which also takes the values $\{+1,-1,0\}$ (see Section \ref{S:f-s-induced}). We can label an involution $\overline t\in G/N$ as being of $+1,-1$ or $0$ type, depending on the value of $\epsilon_{\overline t}(\theta)$. Our second result is:

\begin{Theorem}\label{T:epsilon(theta^G)}
The Frobenius-Schur indicator $\epsilon(\theta{\uparrow^G})$ equals the number of $+1$ type involutions in $G_\theta^*/N$ minus the number of $-1$ type involutions in $G_\theta^*/N$.
\end{Theorem}

We describe the proof of Gallagher's theorem in some detail in Section \ref{S:gallagher}. This will allow us to introduce notation and enable the reader to compare our approach with Gallagher's. In Section \ref{S:f-s-induced} we define the Gow-indicator and prove Theorem \ref{T:epsilon(theta^G)}. We break the proof of Theorem \ref{T:main} into the 2 cases $\theta=\overline\theta$ or $\theta\ne\overline\theta$. Proofs of both cases are given in Section \ref{S:real}.

\section{Gallagher's theorem}\label{S:gallagher}

Our presentation of Gallagher's Theorem is largely as in \cite{Gallagher}, but with some additional details. An alternative approach, detailed in \cite{Navarro}, is to use character triples. This leads to a number of interesting results about character values. It should be possible to prove Theorem \ref{T:main} using similar methods.

We quote the following result from \cite{Gallagher}:

\begin{Lemma}\label{L:Gallagher}
Let $\chi\in\operatorname{Irr}(G)$ and let $H\leq G$ such that $\chi{\downarrow_H}$ is irreducible. Then
$$
\sum_{x\in Hg}|\chi(x)|^2=|H|,\quad\mbox{for each coset $Hg\subseteq G$.}
$$
\end{Lemma}

Let $\operatorname{Cl}(G)$ be the set of conjugacy classes of $G$. For $K\in\operatorname{Cl}(G)$, let $K^o$ be the conjugacy class of inverses of elements of $K$. Also set $K^+$ as the sum of the elements of $K$, taken in the group algebra ${\mathbb C} G$, for all $K\subseteq G$. 

Gallagher begins by defining the following element of $Z({\mathbb C} G)$, where $[x,y]:=x^{-1}y^{-1}xy$ is the commutator of $x$ and $y$:
$$
\begin{aligned}
T:=\sum_{x,y\in G}[x,y]=\sum_{K\in\operatorname{Cl}(G)}\frac{|G|}{|K|}K^{o+}K^+.
\end{aligned}
$$
For $\chi\in\operatorname{Irr}(G)$, let $\omega_\chi:Z({\mathbb C} G)\rightarrow{\mathbb C}$ be the corresponding central character. So $\omega_\chi$ is the algebra homomorphism $\omega_\chi(z)=\chi(z)/\chi(1)$, for all $z\in Z({\mathbb C} G)$, and
$$
\omega_{\chi}(T)=\sum_{K\in\operatorname{Cl}(G)}\frac{|G|}{|K|}\omega_\chi(K^{o+})\omega_\chi(K^+)=\sum_{K\in\operatorname{Cl}(G)}\frac{|G|}{|K|}\frac{|K|^2}{\chi(1)^2}|\chi(k)|^2=\left(\frac{|G|}{\chi(1)}\right)^2.
$$
The last equality follows from the first orthogonality relations. This simplifies to
\begin{equation}\label{E:chi(T)}
\chi(1)\chi(T)=|G|^2,\quad\mbox{for all $\chi\in\operatorname{Irr}(G)$}.
\end{equation}

\begin{proof}[Proof of Gallagher's Theorem]
Recall that $N$ is a normal subgroup of $G$ and $\theta$ is an irreducible character of $N$. The Clifford correspondence is a bijection $\operatorname{Irr}(G\mid\theta)\leftrightarrow\operatorname{Irr}(G_\theta\mid\theta)$ defined by induction-restriction. So we assume that $\theta$ is $G$-invariant, for the rest of this section.

Our assumption gives $\theta(1)\theta{\uparrow^G}=\sum_{\chi\in\operatorname{Irr}(G\mid\theta)}\chi(1)\chi$. This and \eqref{E:chi(T)} implies that
$$
\theta(1)\theta{\uparrow^G}(T)=\sum_{\chi\in\operatorname{Irr}(G\mid\theta)}\chi(1)\chi(T)=|G|^2|\operatorname{Irr}(G\mid\theta)|.
$$
For $x\in G$, notice that $\operatorname{C}_G(\overline{x})=\{y\in G\mid[x,y]\in N\}$. Now $\theta{\uparrow^G}$ vanishes off $N$, and restricts to $|G:N|\theta$ on $N$. So the last displayed formula gives
\begin{equation}\label{E:count}
|\operatorname{Irr}(G\mid\theta)|=\frac{\theta(1)}{|G|\,|N|}\sum_{\substack{x\in G\\y\in\operatorname{C}_G(\overline{x})}}\theta([x,y]).
\end{equation}

Let $x\in G$ and recall that $\theta_x$ is an extension of $\theta$ to $N\langle x\rangle$. All such extensions have the form $\omega\theta_x$, for some $\omega\in\operatorname{Lin}(N\langle x\rangle/N)$. Now $\operatorname{C}_G(\overline{x})$ acts on the extensions of $\theta$ to $N\langle x\rangle$. So given $y\in\operatorname{C}_G(\overline{x})$, there is $\omega_y\in\operatorname{Lin}(\langle\overline x\rangle)$ such that $\theta_x^y=\omega_y\theta_x$. As $\operatorname{C}_G(\overline{x})$ acts trivially on $\operatorname{Lin}(\langle\overline x\rangle)$, $(\omega\theta_x)^y=\omega\theta_x^y=\omega_y(\omega\theta_x)$, for each $\omega\in\operatorname{Lin}(\langle\overline x\rangle)$. So $\omega_y$ does not depend on the choice of extension $\theta_x$. In particular $\omega_{y_1y_2}=\omega_{y_1}\omega_{y_2}$, for all $y_1,y_2\in\operatorname{C}_G(\overline{x})$. This shows that $y\mapsto\omega_y(x)$ is a linear character of $\operatorname{C}_G(\overline{x})$.

We say that the coset $\overline{x}$ is good for $\theta$ if $\theta_x^y=\theta_x$, for all $y\in\operatorname{C}_G(\overline{x})$. So $\overline{x}$ is good for $\theta$ if and only if $y\mapsto\omega_y(x)$ is the trivial character of $\operatorname{C}_G(\overline{x})$. Clearly being good for $\theta$ is independent of the choice $\theta_x$ of extension of $\theta$ to $N\langle x\rangle$.

For $x\in G$ set $X:=\{x^c\mid c\in \operatorname{C}_G(\overline{x})\}$ as the conjugacy class of $\operatorname{C}_G(\overline{x})$ containing $x$. Then $X^+\in Z({\mathbb C}\operatorname{C}_G(\overline{x}))$. As $\sum\limits_{y\in\operatorname{C}_G(\overline{x})}[x,y]=|\operatorname{C}_G(x)|x^{-1}X^+$, we get
$$
\begin{aligned}
\theta(1)\sum_{y\in\operatorname{C}_G(\overline{x})}\theta([x,y])
&=\theta_x(x^{-1})\sum_{y\in\operatorname{C}_G(\overline{x})}\theta_x(x^y)\\
&=|\theta_x(x)|^2\sum_{y\in\operatorname{C}_G(\overline{x})}\omega_y(x)\\
&=
\left\{ 
\begin{array}{ll}
|\operatorname{C}_G(\overline{x})|\,|\theta_x(x)|^2,&\quad\mbox{if $x$ is good for $\theta$.}\\
0,&\quad\mbox{otherwise}.
\end{array}
\right.
\end{aligned}
$$
Summing over all $x\in G$, \eqref{E:count} gives
$$
\begin{aligned}
|\operatorname{Irr}(G\mid\theta)|
&=\frac{1}{|G|\,|N|}\sum_{\overline x\in\overline{G}\,\rm{good}}|\operatorname{C}_G(\overline{x})|\sum_{n\in N}|\theta_x(nx)|^2\\
&=\sum_{\overline x\in\overline{G}\,\rm{good}}\frac{|\operatorname{C}_G(\overline{x})|}{|G|},\quad\mbox{by Lemma \ref{L:Gallagher}, as $\theta_x{\downarrow_N}\in\operatorname{Irr}(N)$}\\
&=\mbox{number of $\theta$-good classes of $G/N$.}
\end{aligned}
$$
\end{proof}

\section{Frobenius-Schur indicator of an induced character}\label{S:f-s-induced}

We continue with our notations for $N$, $\overline G=G/N$, $\theta\in\operatorname{Irr}(N)$, $G_\theta$ and $G_\theta^*$. In this section we prove Theorem \ref{T:epsilon(theta^G)}. This computes $\epsilon(\theta{\uparrow^G})$ in terms of the involutions in $G_\theta^*/N$. Note that $G_\theta\leq G_\theta^*$, $[G_\theta^*:G_\theta]\leq2$ and $G_\theta=G_\theta^*$ if and only if $\theta=\overline\theta$.

Recall that the Frobenius-Schur (F-S) indicator of $\theta$ is
$$
\epsilon(\theta):=\frac{1}{|N|}\sum_{n\in N}\theta(n^2).
$$
Then $\epsilon(\theta)=0$ if $\theta\ne\overline\theta$, $\epsilon(\theta)=+1$, if $\theta$ is the character of a real representation of $N$, and $\epsilon(\theta)=-1$, if $\theta$ has real values but is not the character of a real representation. We extend $\epsilon$ to all generalized characters by linearity.

It is convenient to use the term {\em involution} for a group element of order $1$ or $2$. Let $\Omega(G)$ be the set of involutions in $G$ and let $\Omega(G,N):=\{g\in G\mid g^2\in N\}$ be the set of elements of $G$ whose images are involutions in $\overline G$.

We begin by recalling a generalization of the F-S indicator, which we call the Gow indicator. For all $t\in\Omega(G,N)$, the Gow indicator of $\theta$ (w.r.t. $t$, $\overline t$ or $N\langle t\rangle$) is 
$$
\epsilon_t(\theta):=\frac{1}{|N|}\sum_{n\in N}\theta((nt)^2).
$$
So $\epsilon_t(\theta)=\epsilon(\theta)$, for $t\in N$, and $\epsilon_t(\theta)=\epsilon(\theta{\uparrow^{N\langle t\rangle}})-\epsilon(\theta)$, for $t\in\Omega(G,N)\backslash N$.

R. Gow introduced this indicator in \cite{Gow} and proved the remarkable fact that $\epsilon_t(\theta)\in\{0,-1,1\}$. Moreover $\epsilon_t(\theta)\ne0$ if and only if $\theta^t=\overline\theta$. In \cite{Ichikawa} the $8$ possible non-zero values of $(\epsilon(\theta),\epsilon_t(\theta))$ were identified with the $8$-th roots of unity in ${\mathbb C}$. This determines the element of the Brauer-Wall group of ${\mathbb R}$ which can be associated to the triple $(N,N\langle t\rangle,\theta)$ (the group algebra ${\mathbb C} N\langle t\rangle$ is a ${\mathbb Z}/2{\mathbb Z}$-graded-algebra, with $0$-component ${\mathbb C} N$ and $1$-component ${\mathbb C} N\langle t\rangle\backslash N$).

As we are concerned with reality, we consider the extended Clifford correspondence $\operatorname{Irr}({G\mid\theta})\leftrightarrow\operatorname{Irr}({G_\theta^*\mid\theta})$. This preserves inner-products and F-S indicators. For example $\epsilon(\theta{\uparrow^G})=\epsilon(\theta{\uparrow^{G_\theta^*}})$. So for this section, we will assume that $G=G_\theta^*$, although for the convenience of the reader we state Lemmas \ref{L:real} and \ref{L:non-real} for arbitrary $G$.

\medskip
{\bf Case 1: $\theta$ is real and $G$-invariant.} As $\theta{\uparrow^G}$ vanishes off $N$ and equals $|G:N|\theta$ on $N$, we get
\begin{equation}\label{E:count1}
\epsilon(\theta{\uparrow^G})=\frac{1}{|G|}\sum_{t\in\Omega(G,N)}|G:N|\theta(t^2)=\sum_{\overline t\in\Omega(\overline G)}\epsilon_t(\theta).
\end{equation}
Now for all $t\in\Omega(G,N)$, we have $\epsilon_t(\theta)\ne0$, as $\theta^t=\theta=\overline\theta$. Let $\theta_t$ be one of the two extensions of $\theta$ of $N\langle t\rangle$. If $\epsilon_t(\theta)=\epsilon(\theta)$, then $\theta_t$ is real with F-S indicator $\epsilon(\theta)$. If instead $\epsilon_t(\theta)=-\epsilon(\theta)$, then $\theta_t$ is not real. We say that an involution $\overline t\in\overline G$ is of $\epsilon_t(\theta)$ type for $\theta$, if $\epsilon_t(\theta)=\pm1$. Thus we have proved:

\begin{Lemma}\label{L:real}
Let $\theta$ be real. Then $\epsilon(\theta{\uparrow^G})$ equals the number of \/ $+1$ type involutions in $\overline G_\theta$ minus the number of\/ $-1$ type involutions in $\overline G_\theta$.
\end{Lemma}

\begin{Example}
Take $G=Q_8$, $N=Z(Q_8)$ and $\theta$ the non-trivial linear character of $Z(Q_8)$. Then $G/N$ has 3 non-trivial involutions, each of $-1$ type. So $\epsilon(\theta^G)=1-3=-2$. Indeed $\theta{\uparrow^G}=2\chi$, where $\chi\in\operatorname{Irr}(G)$ has F-S indicator $-1$.
\end{Example}

As an application of Lemma \ref{L:real}, we give a quick proof of the main result in \cite{Richards}:

\begin{Corollary}
Suppose that $|G_\theta/N|$ is odd. Then each real $\theta\in\operatorname{Irr}(N)$ has a unique real extension to $G_\theta$, and there is a unique real character in $\operatorname{Irr}(G\mid\theta)$.
\begin{proof}
We may assume that $\theta$ is $G$-invariant. As $\Omega(G,N)=N$, we get $\sum_{\chi\in\operatorname{Irr}(G\mid\theta)}\epsilon(\chi)\chi(1)=\epsilon(\theta)\theta(1)$. But $\langle\chi{\downarrow_N},\theta\rangle$ is odd, for all $\chi\in\operatorname{Irr}(G\mid\theta)$. So $\epsilon(\chi)=\epsilon(\theta)$, if $\chi=\overline\chi$. We conclude that there is a unique real $\chi\in\operatorname{Irr}(G\mid\theta)$, and moreover, $\chi{\downarrow_N}=\theta$.
\end{proof}
\end{Corollary}

We get a simpler version of Lemma \ref{L:real} for real characters of $2$-defect zero:

\begin{Corollary}
Let $\theta$ be real and of\/ $2$-defect $0$. Then $\epsilon(\theta{\uparrow^G})$ equals the number of involutions in $\overline G_\theta$.
\begin{proof}
Recall that $\theta$ has $2$-defect $0$ if $|N|/\theta(1)$ is an odd integer. Then it is known that $\epsilon(\theta)=+1$. We may assume that $\theta$ is $G$-invariant. Let $t\in\Omega(G,N)\backslash N$ and let $\lambda\in\operatorname{Lin}(N\langle t\rangle/N)$ be non-trivial. Then $\theta_t$ and $\lambda\theta_t$ are the two extensions of $t$ to $N\langle t\rangle$. Now $N$ contains all $2$-regular elements of $N\langle t\rangle$. So $\theta_t$ and $\lambda\theta_t$ belong to the same $2$-block of $N\langle t\rangle$. This $2$-block is real. Thus both $\theta_t$ and $\lambda\theta_t$ are real, and $\overline t$ is of $+$ type.
\end{proof}
\end{Corollary}

G. Navarro and R. Gow have independently proved that with $\theta$ as in the statement of the Corollary, there is a canonical (hence real) extension, say $\chi$, of $\theta$ to $G$. Our Corollary can also be proved using this and Gallagher's result that
$$
\operatorname{Irr}({G\mid\theta})=\{\mu\chi\mid\mu\in\operatorname{Irr}(G/N)\}.
$$

\medskip
{\bf Case 2: $\theta$ has $G$-orbit $\{\theta,\overline\theta\}$ and $\theta\ne\overline\theta$.} So $G_{\theta}$ has index $2$ in $G$, and $\theta^g=\overline\theta$, for all $g\in G-G_{\theta}$. As $\theta{\uparrow^G}$ vanishes off $N$, and restricts to $|G_{\theta}:N|(\theta+\overline\theta)$ on $N$, we get
$$
\epsilon(\theta{\uparrow^G})=\frac{|G_{\theta}:N|}{|G|}\sum_{t\in\Omega(G,N)}(\theta(t^2)+\overline\theta(t^2))=\sum_{\overline t\in\Omega(\overline G)\backslash\overline G_{\theta}}\epsilon_t(\theta).
$$
Notice that if $t\in\Omega(G_{\theta},N)$, then $\theta^t=\theta\ne\overline\theta$. So $\epsilon_t(\theta)=0$, and $t$ does not contribute to $\epsilon(\theta{\uparrow^G})$. We say that $\overline t$ has $0$ type for $\theta$. On the other hand, for $t\in\Omega(G,N)\backslash G_{\theta}$, $\epsilon_t(\theta)$ coincides with the F-S indicator of the real irreducible character $\theta^{N\langle t\rangle}$, and we say that $\overline t\in\Omega(\overline G)\backslash\overline G_{\theta}$ has $\epsilon_t(\theta)$ type, for $\epsilon_t(\theta)=\pm1$. Thus we have proved:

\begin{Lemma}\label{L:non-real}
Let $\theta$ be non-real. Then $\epsilon(\theta{\uparrow^G})$ equals the number of\/ $+1$ type involutions in $\overline G_\theta^*\backslash\overline G_\theta$ minus the number of\/ $-1$ type involutions in $\overline G_\theta^*\backslash \overline G_\theta$.
\end{Lemma}

Notice that $\epsilon(\theta{\uparrow^G})=0$, unless $\overline G=\overline G_\theta^*$ splits over $\overline G_\theta$.

Once again, we get a simpler version for defect zero characters of $N$:

\begin{Corollary}
Let $\theta$ be non-real and of $2$-defect zero. Then:
\begin{enumerate}
\item $\epsilon(\theta{\uparrow^G})$ equals the number of involutions in $\overline G_\theta^*\backslash\overline G_\theta$.
\item Each involution in $\overline G_\theta^*\backslash\overline G_\theta$ is the image of an involution in $G$.
\item $\epsilon(\theta{\uparrow^G})=0$ unless $G_\theta^*$ splits over $G_\theta$.
\item There is $\chi\in\operatorname{Irr}(G\mid\theta)$ such that $\epsilon(\chi)=+1$ and $\langle\chi{\downarrow_N},\theta\rangle$ is odd.
\end{enumerate}
\begin{proof}
We assume, as above, that $G=G_\theta^*$. Now $\theta$ is the unique irreducible character in a non-real $2$-block of $N$ which has defect $0$. Let $t\in\Omega(G,N)\backslash G_\theta$ and set $H:=N\langle t\rangle$. Then $\theta{\uparrow^H}$ is a real irreducible character in a $2$-block $B_t$ of $H$ which has defect $0$. In particular $\epsilon_t(\theta)=\epsilon(\theta{\uparrow^H})=+1$. So $\overline t$ is of $+$ type. Now (1) follows.

Next, let $e$ be an involution in $H$ which generates the extended defect group of $B_t$. We claim that $e\in Nt$. For suppose that $e\in N$. We know that $\theta{\uparrow^H}$ occurs once in $1_C{\uparrow^H}$, where $C=\operatorname{C}_H(e)$. Then by Mackey's formula, $\theta$ occurs once in $1_{C\cap N}{\uparrow^N}$. But $C\cap N$ is just the centralizer of the involution $e\in N$ in $N$. As $\theta$ is projective and irreducible, $\theta$ is also real, by the main result in \cite{Murray}. This proves the claim, which establishes (2).

(3) follows directly from (1) and (2).

Finally, to prove (4), let $S$ be a Sylow $2$-subgroup of $G$. Then $(\theta{\uparrow^G}){\downarrow_S}=m\rho_S$, where $m:=\theta(1)|G:N|/|S|$ is odd and $\rho_S$ is the regular character of $S$. In particular $\langle\theta{\uparrow^G},1_S{\uparrow^G}\rangle=m$ is odd. Now $\theta{\uparrow^G}$ is a real character of $G$. So it follows that there is a real $\chi\in\operatorname{Irr}(G)$ such that $\langle\theta{\uparrow^G},\chi\rangle\langle\chi,1_S{\uparrow^G}\rangle$ is odd. As $\langle\chi,1_S{\uparrow^G}\rangle$ is odd and $\chi$ is real, it follows that $\epsilon(\chi)=+1$.
\end{proof}
\end{Corollary}

Now for arbitrary $\theta\in\operatorname{Irr}(N)$, Lemmas \ref{L:real} and \ref{L:non-real} show that $\epsilon(\theta{\uparrow^G})$ is given by the difference between the numbers of two types of involutions which conjugate $\theta$ to $\overline\theta$ in $\overline G$:

\section{Real characters over an irreducible character of a normal subgroup}\label{S:real}

Let $\operatorname{Irr}_{\mathbb R}(G)$ be the set of real irreducible characters of $G$ and let $\operatorname{Irr}_{\mathbb R}({G\mid\theta})$ be the subset of real characters lying over $\theta$. In place of Gallagher's commutator sum $T$, we consider the following element in the centre of ${\mathbb C} G$:
$$
T_{\mathbb R}:=\sum_{x,y\in G}xy^{-1}xy=\sum_{K\in\operatorname{Cl}(G)}|C_G(g_K)|(K^+)^2.
$$
Now for each $\chi\in\operatorname{Irr}(G)$ we have
$$
\omega_\chi(T_{\mathbb R})=\sum_{K\in\operatorname{Cl}(G)}|C_G(g_K)||K|^2\frac{\chi(g_K)^2}{\chi(1)^2}=
\left\{
\begin{array}{cl}
 \left(\frac{|G|}{\chi(1)}\right)^2,&\quad\mbox{if $\chi=\overline\chi$,}\\
 0,&\quad\mbox{if $\chi\ne\overline\chi$.}
\end{array}\right.
$$
This gives (c.f. \eqref{E:chi(T)})
\begin{equation}\label{E:chi(R)}
\chi(1)\chi(T_{\mathbb R})=\epsilon(\chi)^2|G|^2,\quad\mbox{for all $\chi\in\operatorname{Irr}(G)$.}
\end{equation}

We now prove the first half of our main theorem:

\begin{proof}[Proof of Theorem \ref{T:main} for real characters]
Suppose that $\theta\in\operatorname{Irr}(N)$ is real. Then the Clifford correspondence preserves reality and we can and do assume that $\theta$ is $G$-invariant. Thus
$$
\theta(1)\theta{\uparrow^G}(T_{\mathbb R})=\sum_{\chi\in\operatorname{Irr}(G\mid\theta)}\chi(1)\chi(T_{\mathbb R})=|G|^2|\operatorname{Irr}_{\mathbb R}(G\mid\theta)|.
$$

Each $\overline x\in\overline G$ has both a centralizer $\operatorname{C}_G(\overline x):=\operatorname{N}_G(Nx)$ and an extended centralizer $\operatorname{C}_G^*(\overline x):=\operatorname{N}_G(Nx\cup Nx^{-1})$ in $G$. Set $\operatorname{I}_G(\overline x):=\{y\in G\mid xy^{-1}xy\in N\}$. Note that $\operatorname{I}_G(\overline x)$ is empty, unless $\overline x$ is real in $\overline G$. With this notation, we have
\begin{equation}\label{E:real_count}
|\operatorname{Irr}_{\mathbb R}(G\mid\theta)|=\frac{1}{|G|\,|N|}\sum_{x\in G\atop\overline{x}{\,\rm real}}\sum_{y\in\operatorname{I}_G(\overline x)}\theta(1)\theta(xy^{-1}xy).
\end{equation}

For $x\in G$, such that $\overline x$ is real in $\overline G$, we may choose $e\in\operatorname{I}_G(\overline x)$ so that $\operatorname{I}_G(\overline x)=\operatorname{C}_G(\overline x)e$. As before, $\theta_x$ is an extension of $\theta$ to $N\langle x\rangle$. Now $\theta_x^e$ is another extension of $\theta$ to $N\langle x\rangle$. Following Gallagher's argument, we have
$$
\begin{aligned}
\theta(1)\sum_{y\in\operatorname{I}_G(\overline x)}\theta(xx^y)
&=\theta_x(x)\sum_{c\in\operatorname{C}_G(\overline x)}\theta_x^e(x^c)\\
&=\left\{ 
    \begin{array}{ll}
      |\operatorname{C}_G(\overline{x})|\,\theta_x(x)\theta_x^e(x),&\quad\mbox{if $\overline x$ is good for $\theta$.}\\
      0,&\quad\mbox{otherwise}.
    \end{array}
  \right.
\end{aligned}
$$

Suppose now that $\overline x$ is good for $\theta$. As $\overline\theta_x$ is an extension of $\overline\theta=\theta$ to $N\langle x\rangle$, we may write $\theta_x^e=\sigma\overline\theta_x$, for some $\sigma\in\operatorname{Lin}(\langle \overline x\rangle)$. Let $\lambda\in\operatorname{Irr}(\langle \overline x\rangle)$. As $e$ inverts $\overline x$, also $\lambda^e=\overline\lambda$. So $(\lambda\theta_x)^e=\overline\lambda\theta_x^e =\sigma\overline{\lambda\theta_x}$. This shows that $\sigma$ is independent of the choice of $e$ and of the extension $\theta_x$ of $\theta$ to $N\langle x\rangle$.

As $e^2\in\operatorname{C}_G(\overline x)$, and $\overline x$ is good for $\theta$, we have
$$
\theta_x=\theta_x^{e^2}=(\sigma\overline\theta_x)^e=\overline\sigma^2\theta_x.
$$
So $\overline\sigma^2$ is trivial, whence $\sigma=\overline\sigma$. So if $\sigma$ is non-trivial, $\operatorname{ker}(\sigma)$ has index $2$ in $\langle\overline x\rangle$, in which case $\sigma(x)=-1$. This shows that $\sigma(x)=1$, if $\theta_x^e=\overline\theta_x$, and $\sigma(x)=-1$, if $\theta_x^e\ne\overline\theta_x$\footnote{In the notation of Section \ref{S:f-s-induced}, $\sigma(x)=(-1)^{1+\epsilon_e(\theta_x)^2}$}.

As $x$ is good for $\theta$, $nx$ is good for $\theta$ and $\sigma(x)=\sigma(nx)$, for all $n\in N$. So
$$
\sum_{n\in N}\theta_x(xn)\theta_x^e(xn)=\sigma(x)\sum_{n\in N}|\theta_x(xn)|^2=|N|\sigma(x).
$$

Finally $\sigma(\overline x)=\sigma(\overline y)$, whenever $\overline x$ and $\overline y$ are conjugate in $\overline G$. So
$$
|\operatorname{Irr}_{\mathbb R}(G\mid\theta)|=\frac{1}{|G|\,|N|}\sum_{\overline x{\,\rm good}}|\operatorname{C}_{\overline G}(\overline x)|\sigma(\overline x)=\sum\sigma(\overline x),
$$
where $\overline x$ ranges over a set of representatives for the real conjugacy classes of $\overline G$ which are good for $\theta$.
\end{proof}

Note that if $\overline x$ is $2$-regular, then it is good for $\theta$. For, we can choose $\theta_x$ to be the unique real extension of $\theta$ to $N\langle x\rangle$ (c.f. Richards's Theorem). Then for all $y\in\operatorname{C}_G(\overline x)$, $\theta_x^y$ is a real extension of $\theta$ to $N\langle x\rangle$. So $\theta_x^y=\theta_x$, by the uniqueness of $\theta_x$. Moreover, $\sigma$ is trivial, as it is a real character of the odd order group $\langle\overline x\rangle$. So
$$
\sigma(\overline x)=+1,\quad\mbox{if $\overline x$ is a $2$-regular element of $\overline G$.}
$$

Suppose instead that $\overline x^2=\overline1$. Then $\theta_x^e=\theta_x$, as $e\in\operatorname{C}_G(\overline x)$. So $\sigma(\overline x)=1$ if $\theta_x$ is real and $\sigma(\overline x)=-1$, if $\theta_x$ is not real.

Our corollary could be proved using Clifford theory and Brauer's permutation lemma:

\begin{Corollary}
Let $Z$ be a central subgroup of $G$ of order $2$, and let $\zeta$ be the non-trivial linear character of $Z$. Then the number of real irreducible characters in $\operatorname{Irr}(G\mid\zeta)$ is the number of real splitting classes of\/ $\overline G$ minus twice the number of real splitting classes of\/ $\overline G$ which are the image of non-real classes of\/ $G$. 
\begin{proof}
Let $Z=\langle z\rangle$. There are two possibilites for each $g\in G$. Suppose first that $g$ is not $G$-conjugate to $zg$, which means $\operatorname{C}_G(\overline g)=\operatorname{C}_G(g)$. Then $\overline g$ is $\zeta$-good as $\zeta_g^c=\zeta_g$, for all $c\in\operatorname{C}_G(gZ)$. Also the $\overline G$-conjugacy class of $\overline g$ is the image of two conjugacy classes of $g$; the one containing $g$ and the other containing $zg$. Each $\chi\in\operatorname{Irr}(G\mid\zeta)$ satisfies $\chi(zg)=-\chi(g)$.

Suppose then that $[\operatorname{C}_G(\overline g):\operatorname{C}_G(g)]=2$. Choose $x\in G$ such that $g^x=zg$. Then
$$
\zeta_g^x(g)=\zeta_g(zg)=-\zeta_g(g)\ne\zeta_g(g),\quad\mbox{as $\zeta_g(g)\ne0$}.
$$
So $gZ$ is not $\zeta$-good. Also the $\overline G$-conjugacy class of $\overline g$ is the image of one conjugacy class of $g$. Each $\chi\in\operatorname{Irr}(G\mid\zeta)$ satisfies $\chi(zg)=\chi(g)=0$.

The previous two paragraphs show that the number of $\zeta$-good conjugacy classes of $G/Z$ equals the number of `splitting' conjugacy classes of $G/Z$, which is what we expect.

Now suppose that $gN$ is both $\zeta$-good and real, in $G/Z$. So there is $e\in G$ such that $g^e\in\{g^{-1},zg^{-1}\}$. Suppose that $g$ is not real in $G$. Then $g^e=zg^{-1}$. In that case $\zeta_g^e(g)=\zeta(zg^{-1})=-\overline\zeta_g(g)$. So $\sigma(g)=-1$. The other possibility is that $g$ is real in $G$. So we may assume that $g^e=g^{-1}$ and then $\zeta_g^e=\overline\zeta_g$. So $\sigma(g)=1$.
\end{proof}
\end{Corollary}

\medskip
{\bf Example:} Suppose that $G=2.A_8$ and $\overline G=A_8$. Then $\overline G$ has $9$ splitting classes and $9$ faithful irreducible characters. In addition, $\overline G$ has $5$ real splitting classes. However one of these $6B$ is the image of two non-real classes of $G$. So $G$ has $5-2=3$ real faithful irreducible characters.

\medskip
{\bf Example:}
Let $G=\operatorname{GL}(2,3)$ and $N=\operatorname{Z}(G)$, cyclic of order $2$. So $G/N\cong S_4$. Now $G$ has three irreducible characters lying over the non-trivial character $\zeta$ of $N$, but only one of these characters is real. So $G/N$ has three $\zeta$-good conjugacy classes, all of which are real. We deduce that $\sigma(x)=1$ for two classes and $\sigma(x)=-1$ for the remaining class.

\medskip
We finish by proving the second half of Theorem \ref{T:main}:

\begin{proof}[Proof of Theorem \ref{T:main} for non-real characters] Suppose that $\theta\in\operatorname{Irr}(G)$ is not real. By the extended Clifford correspondence, we can and do assume that $\theta$ has $G$-orbit $\{\theta,\overline\theta\}$. So $G=G_\theta^*$ and $M:=G_\theta$ has index $2$ in $G$.

For $\chi\in\operatorname{Irr}(G\mid\theta)$, we have $\langle\chi{\downarrow_N},\theta\rangle=\frac{\chi(1)}{2\theta(1)}$. So
$$
2\theta(1)\theta{\uparrow^G}=\sum_{\chi\in\operatorname{Irr}(G\mid\theta)}\chi(1)\chi.
$$
Also $2\theta{\uparrow^G}$ vanishes off $N$ and restricts to $|G:N|(\theta+\overline\theta)$ on $N$.

Let $T_{\mathbb R}$ be as in the previous section. Applying $2\theta(1)\theta{\uparrow^G}$ to $T_{\mathbb R}$ we get
\begin{equation}\label{E:real}
|\operatorname{Irr}_{\mathbb R}(G\mid\theta)|=\frac{\theta(1)}{2|G|\,|N|}\sum_{x\in G\atop\overline x{\rm\,real\,}}\sum_{y\in\operatorname{I}_G(\overline x)}(\theta(xx^y)+\overline\theta(xx^y)).
\end{equation}

Now $\operatorname{I}_G(\overline x)$ is a union of one or two cosets of $\operatorname{C}_M(\overline x)$, each of which is contained in $M$ or in $G\backslash M$. In addition, either $x\in M$ or $x\in G\backslash M$. This gives three types of cosets $\operatorname{C}_M(\overline x)e\subseteq\operatorname{I}_G(\overline x)$:
\begin{itemize}
\item[Type 1:] $x\in M$ and $e\in\operatorname{I}_G(\overline x)\backslash M$.
\item[Type 2:] $x\in M$ and $e\in\operatorname{I}_M(\overline x)$.
\item[Type 3:] $x\in G\backslash M$.
\end{itemize}
As we shall see, only the first type contributes to $|\operatorname{Irr}_{\mathbb R}(G\mid\theta)|$.

\medskip
{\bf Type 1:} Suppose that $x\in M$ and $e\in\operatorname{I}_G(\overline x)\backslash M$. So $\overline x^e=\overline x^{-1}$ and $\theta^e=\overline\theta$. We evaluate the contribution of the coset $\operatorname{C}_M(\overline x)e$ to the right hand side of \eqref{E:real}. Choose an extension $\theta_x$ of $\theta$ to $N\langle x\rangle$. Now $\theta_x^e$ is an extension of $\overline\theta$ to $N\langle x\rangle$. So $\theta_x^e=\sigma\overline\theta_x$, for some $\sigma\in\operatorname{Lin}(\langle\overline x\rangle)$, where $\sigma$ is independent of $\theta_x$. Moreover, for $c\in\operatorname{C}_G(\overline x)$, $\theta_x^c=\omega_c\theta_x$ for some $\omega_c\in\operatorname{Lin}(\langle\overline x\rangle)$, where $\omega_y$ is independent of $\theta_x$. So $\theta_x^{ce}=(\omega_c\theta_x)^e=\overline\omega_c\sigma\overline\theta_x$, using the fact that $e$ inverts $\operatorname{Lin}(\langle\overline x\rangle)$. Thus
$$
\begin{aligned}
\theta(1)\sum_{c\in\operatorname{C}_M(\overline{x})}(\theta(xx^{ce})+\overline\theta(xx^{ce}))
=\left(\sum_{c\in\operatorname{C}_{M}(\overline{x})}(\theta_x(x)\overline\omega_c(x)\sigma(x)\overline\theta_x(x)+\overline\theta_x(x)\omega_c(x)\overline\sigma(x)\theta_x(x))\right)\\
=\left\{ 
\begin{array}{cl}
|\operatorname{C}_{M}(\overline{x})|\,|\theta_x(x)|^2(\sigma(x)+\overline{\sigma(x)}),&\quad\mbox{if $x$ is good for $\theta$ in $M$.}\\
0,&\quad\mbox{otherwise}.
\end{array}
\right.
\end{aligned}
$$
Now suppose that $x$ is good for $\theta$ in $M$. Then $\theta_x^{e^2}=\theta_x$, as $e^2\in\operatorname{C}_M(\overline x)$. So $\theta_x=\theta_x^{e^2}=(\sigma\overline\theta_x)^e=\overline\sigma^2\theta_x$. We deduce that $\sigma^2$ is trivial and hence that $\sigma$ is real. So $\sigma(x)=1$, if $\theta_x^e=\overline\theta_x$ and $\sigma(x)=-1$ if $\theta_x^e\ne\overline\theta_x$. Therefore the contribution of $\operatorname{C}_M(\overline{nx})e$ to \eqref{E:real}, as $n$ ranges over $N$, is
$$
|\operatorname{C}_{M}(\overline{x})|\sum_{n\in N}|\theta_x(nx)|^2(\sigma(nx)+\overline{\sigma(nx)}))=2|\operatorname{C}_{M}(\overline{x})||N|\sigma(x).
$$

\medskip
{\bf Type 2:} Suppose that $x\in M$ and $e\in\operatorname{I}_M(\overline x)$. So $\overline x^e=\overline x^{-1}$ and $\theta^e=\theta$. We evaluate the contribution of $\operatorname{C}_M(\overline x)e$ to \eqref{E:real}. Choose an extension $\theta_x$ of $\theta$ to $N\langle x\rangle$. Then for $y\in\operatorname{C}_M(\overline x)e$, we have $\theta_x^y=\omega_y\theta_x$, where $\omega_y\in\operatorname{Lin}(\langle\overline x\rangle)$ is independent of $\theta_x$. Taking complex conjugates, we get $\overline\theta_x^y=\overline\omega_y\overline\theta_x$. So
$$
\begin{aligned}
\theta(1)\sum_{y\in\operatorname{C}_M(\overline{x})e}(\theta(xx^y)+\overline\theta(xx^y))
=\left(\sum_{y\in\operatorname{C}_M(\overline{x})e}(\theta_x(x)\omega_y(x)\theta_x(x)+\overline\theta_x(x)\overline\omega_y(x)\overline\theta_x(x))\right)\\
=\left\{ 
\begin{array}{cl}
|\operatorname{C}_M(\overline{x})|\left(\theta_x(x)^2+\overline\theta_x(x)^2)\right),&\quad\mbox{if $x$ is good for $\theta$ in $M$.}\\
0,&\quad\mbox{otherwise}.
\end{array}
\right.
\end{aligned}
$$
Suppose that $x$ is good for $\theta$. As $\theta_x{\downarrow_N}=\theta$, but $\theta\ne\overline\theta$, we get $\langle(\theta_x^2+\overline\theta_x^2){\downarrow_N},1_N\rangle=0$. So
$$
\sum_{n\in N}(\theta_x^2(nx)+\overline\theta_x^2(nx))=0.
$$
We conclude that the contribution of $\operatorname{C}_M(\overline{nx})e$ to \eqref{E:real}, as $n$ ranges over $N$, is zero.

\medskip
{\bf Type 3:} We finally consider the case that $x\in G\backslash M$ and $e\in\operatorname{I}_G(\overline x)$. Clearly we may assume that $e\in M$. So $\overline x^e=\overline x^{-1}$, $\theta^e=\theta$ and $\operatorname{I}_G(\overline x)=\operatorname{C}_G(\overline x)e$.

Now $\theta^x=\overline\theta\ne\theta$. So we cannot extend $\theta$ to $N\langle x\rangle$. On the other hand, $\theta^{x^2}=\theta$. So we can choose an extension $\theta_{x^2}$ of $\theta$ to $N\langle x^2\rangle$. Now $\theta_{x^2}^x$ is an extension of $\overline\theta$ to $N\langle x^2\rangle$. In particular $\theta_{x^2}^x\ne\theta_{x^2}$. Set $\hat\theta_x:=\theta_{x^2}{\uparrow^{N\langle x\rangle}}$. Then $\hat\theta_x$ is an irreducible character of $N\langle x\rangle$ which vanishes off $N\langle x^2\rangle$ and restricts to $\theta+\overline\theta$ on $N$.

Next $X:=\{x^{ce}\mid c\in\operatorname{C}_G(\overline{x})\}$ is the conjugacy class of $\operatorname{C}_G(\overline{x})$ containing $x^e$. So $X^+$ is in the centre of ${{\mathbb C}\operatorname{C}_G(\overline{x})}$. Moreover $|X|=|\operatorname{C}_G(\overline{x}):\operatorname{C}_G(x)|$ and $\sum\limits_{y\in\operatorname{C}_G(\overline{x})}xx^{ey}=|\operatorname{C}_G(x)|xX^+$.

Now for $y\in\operatorname{I}_G(\overline{x})$, we have $xx^y\in N\leq N\langle x^2\rangle$. So $\theta(xx^y)+\overline\theta(xx^y)=\hat\theta_x(xx^y)$. Thus
$$
\theta(1)\sum_{y\in\operatorname{I}_G(\overline{x})}(\theta(xx^y)+\overline\theta(xx^y))=
\frac{1}{2}\hat\theta_x(x)\sum_{c\in\operatorname{C}_{G}(\overline{x})}\hat\theta_x(x^{ec})=0,
$$
as $\hat\theta_x$ vanishes off $N\langle x^2\rangle$. So the contribution of $\operatorname{I}_G(\overline x)$ to \eqref{E:real} is zero. 

\medskip
We can now compute $|\operatorname{Irr}_{\mathbb R}(G\mid\theta)|$. Note that $\sigma(x)=\sigma(y)$, whenever $\overline x$ and $\overline y$ are conjugate in $\overline G$. Now suppose that $\overline x\in\overline M$ and $(\overline x,\theta)^e=(\overline x^{-1},\overline\theta)$, for some $e\in G$. There are two cases. The first is that the conjugacy class of $\overline G$ which contains $\overline x$ is a single real conjugacy class $K$ of $\overline M$. Then $|\operatorname{C}_G(\overline x)|=2|\operatorname{C}_M(\overline x)|$. So $\sum_{\overline x\in K}\frac{2|\operatorname{C}_M(\overline x)|}{|G|}=1$. The second case is that the conjugacy class $K$ of $\overline G$ containing $\overline x$ is a union of two non-real conjugacy classes of $\overline M$. Then $|\operatorname{C}_G(\overline x)|=|\operatorname{C}_M(\overline x)|$. So $\sum_{\overline x\in K}\frac{2|\operatorname{C}_M(\overline x)|}{|G|}=2$. We conclude that
$$
|\operatorname{Irr}_{\mathbb R}(G\mid\theta)|=\sum_{\overline x{\rm\,good}}\frac{2|\operatorname{C}_M(\overline x)|}{|G|}\sigma(x)=\sum\sigma(\overline x),
$$
where $\overline x$ ranges over a set of representatives for the conjugacy classes of $\overline M$ which are good for $\theta$, such that $\overline x$ is inverted by some element of $G\backslash M$.
\end{proof}

\section*{Acknowledgment}
Parts of this paper were written while the author visited the University of Valencia in September 2022. We gratefully acknowledge the financial support provided by the Department of Mathematics, University of Valencia and thank Gabriel Navarro for his encouragement. We received valuable input from Rod Gow and Benjamin Sambale on this paper.

\end{document}